\documentclass[runningheads]{llncs}
\usepackage[T1]{fontenc}
\usepackage{graphicx}
\usepackage{color}

\usepackage{mathtools,amsfonts,mathrsfs,amssymb,stmaryrd}
\usepackage[dvipsnames]{xcolor}
\usepackage{tikz}
\usepackage{standalone}
\usepackage[justification=centering]{caption}
\usepackage[linesnumbered,ruled,vlined]{algorithm2e}
\SetKwProg{Fn}{Function}{:}{}
\usepackage{comment}

\newcommand\Lshape[2]{
\draw[#1] #2 -- ++(-1, 0) -- ++(0, 1) -- ++(2,0) -- ++(0, -2) -- ++(-1, 0) -- cycle
}
\newcommand{\move}[3]{
\draw[->, >=stealth, color = #1, very thick] #2 -- #3
}
\newcommand{\step}[2]{
\draw[->, >=stealth, thick, #1] #2-- ++(1,0)
}

\newcommand{\trigrid}[1]{
    \foreach \k in {1, ..., #1}
    \draw (\k-1,#1-\k) grid (#1,#1-\k+1);
    }
    
\newcommand{\TriangShape}[2]{
    \draw[#2] (0, #1) -- (#1, #1) -- (#1, 0);
    \foreach \k in {1, ..., #1}
    \draw[#2] (#1 - \k +1, \k -1) -- (#1 - \k, \k -1) -- (#1 - \k, \k);
}

\newcommand{\Z}{\mathbb{Z}}
\newcommand{\N}{\mathbb{N}}
\newtheorem{alg}[algocf]{Algorithm}
\newcommand{\spacc}{\vspace{-0.1cm}}

\begin{document}
\title{Triangle solitaire}

\author{Ville Salo\inst{1}\orcidID{0000-0002-2059-194X} \and
Juliette Schabanel\inst{2}}
\authorrunning{V. Salo \and J. Schabanel}
%
\institute{University of Turku, Turku, Finland \and
École Normale Supérieure PSL, 45 rue d'Ulm, 75005 Paris, France}
\maketitle              
\begin{abstract}
The solitaire of independence is a groupoid action resembling the classical $15$-puzzle, which gives information about independent sets of coordinates in a totally extremally permutive subshift. We study the solitaire with the triangle shape, which corresponds to the spacetime diagrams of bipermutive cellular automata with radius $1/2$. We give a polynomial time algorithm that puts any finite subset of the plane in normal form using solitaire moves, and show that the solitaire orbit of a line of consecutive ones -- the line orbit -- is completely characterised by the notion of a fill matrix. We show that the diameter of the line orbit under solitaire moves is cubic.

\end{abstract}

\section{Introduction}

Multidimensional symbolic dynamics studies sets of vertex-labellings of the lattices $\Z^d$
under local constraints (forbidden patterns), mainly for $d \geq 2$ (the one-dimensional case
being somewhat different in nature \cite{1995-LM}). The basic object in this theory is the
subshift of finite type or SFT, namely the set of $X \subset A^{\Z^2}$ defined by a finite alphabet $A$
and a finite set of finite forbidden patterns, which may not appear anywhere in the configurations
$x \in X$.

A typical phenomenon in this theory is undecidability, namely most basic questions, like whether a given pattern appears in the subshift or whether it is even nonempty, are undecidable \cite{Wa61,Be66}. Natural questions can be high in the arithmetical hierarchy \cite{JeVa15a} (or even analytic). Similarly, the values of invariants such as entropies \cite{HoMe10} and periods \cite{JeVa15} 
are best described by recursion theory and notions from theoretical computer science. In contrast, in the one-dimensional
case simple-sounding problems tend to be solvable (with some exceptions), and there are often combinatorial and algebraic descriptions for values of invariants.

It is an interesting quest to try to find conditions on SFTs that make them less complicated and
more like their one-dimensional counterparts. 
One such class are the TEP subshifts introduced in \cite{Sa22}. Here, we restrict to TEP subshifts defined by a triangle-shaped rule.
It was shown in \cite{Sa22} that these subshifts have at least two properties that set them apart from
general SFTs, namely they have computable languages and admit a kind of uniform measure.

The idea of a TEP subshift arises from the theory of cellular automata. Indeed a two-dimensional
TEP subshift can be seen as the spacetime subshift of a particular kind of cellular automaton. We concentrate here on the triangle shape $\{(1,1),(0,1),(1,0)\}$, equivalently the spacetime diagrams of radius-$\frac12$ bipermutive cellular automaton (with $(0,-1)$ as the arrow of time and $\{-1, 0\}$ as the neighbourhood), studied for example in \cite{MoBo97,Pi05,Sa07,SaTo13}. For example the spacetime subshift of the XOR cellular automaton with alphabet $\{0,1\}$ and local rule $f(a, b) = a \oplus b$, also known as the Ledrappier \cite{Le78} or three-dot subshift, is TEP with this shape.


Specifically, what we study here are the independent sets of a TEP subshift with the triangle shape $\{(0,0), (0,1), (1,0)\} \subset \Z^2$, namely sets $A \subset \Z^2$ having the property that one can freely choose their content. While it remains open how to characterise such sets, we give here a complete algorithm for putting such a set in normal form in terms of the solitaire process introduced in \cite{Sa22}, which reduces the problem of independence to independence of sufficiently disjoint triangular areas.

Our main interest is in the \emph{line orbit}, namely the solitaire orbit of the line of $n$ horizontally adjacent cells. We show that the line orbit corresponds exactly to the fill matrices studied previously by Gerhard Kirchner in \cite{OEISfillmatrices}. 

Letting $(V, E)$ denote the graph with this orbit as its nodes and $E$ as the solitaire moves, from the connection with fill matrices we immediately obtain a polynomial algorithm for checking whether a subset of the plane belongs to $V$ (Algorithm~\ref{algo:ident orbit}), and that the number of vertices is between $(c_1 n)^n$ and $(c_2 n)^n$ for some constants $c_1, c_2$ (Theorem~\ref{thm:orbit size}). We show how to find connecting paths between elements of $V$ in polynomial time (Algorithm~\ref{alg:path}), and show that the diameter of $(V, E)$ is $\Theta(n^3)$ (Theorem~\ref{thm:diameter}). 

Algorithm~\ref{algo:ident orbit} gives a positive answer to the last two subquestions of Question~5.36 in \cite{Sa22}.

\section{Definitions}
\label{sec:Definitions}

The action we consider in this paper is the solitaire move with the triangle shape $T = \{(0,1); (1, 1); (1, 0)\}$ on elements of $\{0, 1\}^{\Z^2}$. This action consists in arbitrarily permuting the three patterns depicted in Figure \ref{fig:triangle_action} at some coordinates where one of them appears in the pattern.
\begin{figure}[ht!]
    \centering
    \begin{tikzpicture}
    \fill[color = gray!50] (0,1) rectangle (2, 2);
    \draw (1,1) rectangle (2,2);
    \draw[color = red, thick] (1, 1) -- ++(-1, 0) -- ++(0, 1) -- ++ (2, 0) -- ++(0, -2) -- ++(-1, 0) -- cycle ;
   
    \draw[<->, >=stealth, thick] (2.5, 1)--(3.5, 1);
   
    \fill[color = gray!50] (4,1) rectangle (5, 2);
    \fill[color = gray!50] (5,0) rectangle (6, 1);
    \draw (5,1) rectangle (6,2);
    \draw[color = red, thick] (5, 1) -- ++(-1, 0) -- ++(0, 1) -- ++ (2, 0) -- ++(0, -2) -- ++(-1, 0) -- cycle ;
   
    \draw[<->, >=stealth, thick] (6.5, 1)--(7.5, 1);
   
    \fill[color = gray!50] (9,0) rectangle (10, 2);
    \draw (9,1) rectangle (10,2);
    \draw[color = red, thick] (9, 1) -- ++(-1, 0) -- ++(0, 1) -- ++ (2, 0) -- ++(0, -2) -- ++(-1, 0) -- cycle ;
\end{tikzpicture}
    \caption{The action of the triangle shape. Grey denotes $1$, white denotes $0$.}
    \label{fig:triangle_action}
\end{figure}
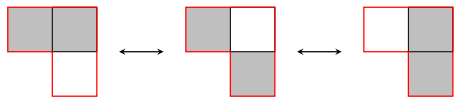
\spacc{}

We call this a \emph{solitaire move}, or more specifically a \emph{triangle move}. The \emph{orbit} of a pattern $P$, denoted $\gamma(P)$ is the set of patterns reachable from it using the triangle moves. In what follows, we only consider patterns with finitely many $1$s. 

The solitaire arises from the study of \emph{TEP subshifts} \cite{Sa22}. We omit the general definition, and give this only for the triangle shape $T$: Let $\Sigma$ be a finite alphabet and let $R \subset \Sigma^T$ satisfy the following, for any $\{a, b, c\} = T$: for all $p \in \Sigma^{\{a, b\}}$ there exists a unique $s\in\Sigma$ such that $p \sqcup (c \mapsto s) \in R$, where $\sqcup$ denotes disjoint union of patterns. Then $R$ is called a \emph{TEP family}. The set $X_R \subset \Sigma^{\Z^2}$, consisting of all $x \in \Sigma^{\Z^2}$ such that for all $\vec v \in \Z^2$ the pattern $t \mapsto x_{\vec v + t}$ is in $R$, is called a \emph{triangular TEP subshift}.

Triangular TEP subshifts are precisely the spacetime subshifts (sets of spacetime diagrams) of bipermutive cellular automata, namely if $f : \Sigma^2 \to \Sigma$ satisfies that $a \mapsto f(a, b)$ and $a \mapsto f(b, a)$ are bijective for all $b \in \Sigma$, then it is easy to see that the patterns $(a, b, f(a, b)) \in \Sigma^{((0,1), (1,1), (1,0))}$ form a TEP family (and the converse holds as well).

\section{Characterisation of the orbits}

In this section, we give a normal form for each $\gamma$-orbit and show a simple way to determine the orbit to which a given pattern belongs.

\subsection{Notations and first result}

The \emph{neighbourhood} of a point $x$ is depicted Figure \ref{fig:nbh}, it corresponds to the points which can be involved in a triangle move with $x$. The neighbourhood of a pattern is the union of the neighbourhoods of its points. Two patterns A and B \emph{touch} if $A\cap N(B) \neq \varnothing$ or $B\cap N(A) \neq \varnothing$.
\begin{figure}[ht]
    \centering
    \begin{tikzpicture}
    \fill[color = cyan!40] (0, 1) -- (0, 3) -- (2, 3) -- (2, 2) -- (1, 2)-- (1, 1) -- cycle;
    \fill[color = cyan!40] (1, 0) -- (3, 0) -- (3, 2) -- (2, 2) -- (2, 1)-- (1, 1) -- cycle;
    \fill[color=orange!75] (1, 1) -- (2, 1) -- (2, 2) -- (1, 2) -- cycle;
    \draw (0,0) grid (3,3);
\end{tikzpicture}
    \caption{The neighbours of the orange cell are the blue ones.}
    \label{fig:nbh}
\end{figure}
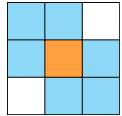
\spacc{}

Let $T_n = \{(a, b) \in \{0,\ldots,n-1\}^2 \;|\; a + b \geq n-1\}$ denote the size-$n$ triangle. By its \emph{edges} we refer to the intersections of the edges of its convex hull with the lattice $\Z^2$.

\begin{proposition}
\label{prop:lines}
    For every $n$, the three edges of $T_n$ are in the same orbit.
\end{proposition}
\begin{proof}
    The first line of Figure \ref{fig:lines_transformation} explains by example how to transform the horizontal line into the diagonal, and the second one how to transform the diagonal into the vertical line (of course this is just a rotated inverse of the first transformation).
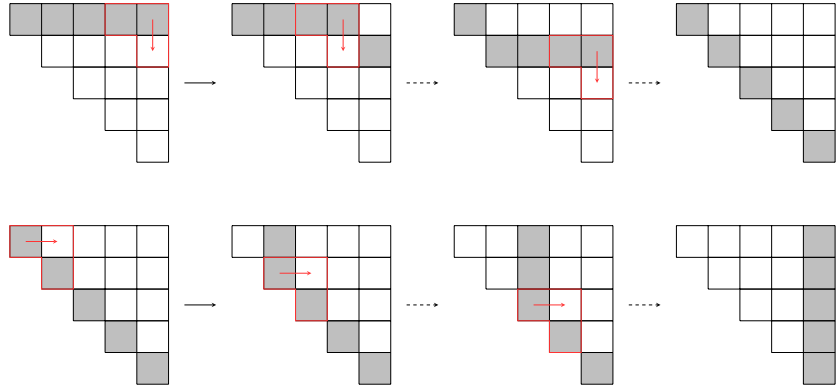
\begin{figure}[ht]
    \centering
    \begin{tikzpicture}
    \begin{scope}
    \fill[color = gray!50] (0, 4) -- (0, 5) -- (5, 5) -- (5, 4) -- cycle;
    \trigrid{5};
    \draw[color = red!80, thick] (4, 4) -- ++(-1, 0) -- ++(0, 1) -- ++ (2, 0) -- ++(0, -2) -- ++(-1, 0) -- cycle ;
    \draw[->, >=stealth, color = red!80, thick] (4.5, 4.5) -- (4.5, 3.5);
    \end{scope}
    
    \draw[->, >=stealth, thick] (5.5, 2.5)--(6.5, 2.5);
    
    \begin{scope}[xshift = 7cm]
    \fill[color = gray!50] (0, 4) rectangle (4, 5);
    \fill[color = gray!50] (4,3) rectangle (5, 4);
    \trigrid{5};
    \draw[color = red!80, thick] (3, 4) -- ++(-1, 0) -- ++(0, 1) -- ++ (2, 0) -- ++(0, -2) -- ++(-1, 0) -- cycle ;
    \draw[->, >=stealth, color = red!80, thick] (3.5, 4.5) -- (3.5, 3.5);
    \end{scope}
    
    \draw[->, >=stealth, thick, dashed] (12.5, 2.5)--(13.5, 2.5);
    
    \begin{scope}[xshift = 14cm]
    \fill[color = gray!50] (0, 4) rectangle (1, 5);
    \fill[color = gray!50] (1,3) rectangle (5, 4);
    \trigrid{5};
    \draw[color = red!80, thick] (4, 3) -- ++(-1, 0) -- ++(0, 1) -- ++ (2, 0) -- ++(0, -2) -- ++(-1, 0) -- cycle ;
    \draw[->, >=stealth, color = red!80, thick] (4.5, 3.5) -- (4.5, 2.5);
    \end{scope}
    
    \draw[->, >=stealth, thick, dashed] (19.5, 2.5)--(20.5, 2.5);
    
    \begin{scope}[xshift = 21cm]
    \fill[color = gray!50] (0, 4) rectangle (1, 5);
    \fill[color = gray!50] (1, 3) rectangle (2, 4);
    \fill[color = gray!50] (2, 2) rectangle (3, 3);
    \fill[color = gray!50] (3, 1) rectangle (4, 2);
    \fill[color = gray!50] (4, 0) rectangle (5, 1);
    \trigrid{5};
    \end{scope}
    
    \begin{scope}[yshift = -7cm]
    \fill[color = gray!50] (0, 4) rectangle (1, 5);
    \fill[color = gray!50] (1, 3) rectangle (2, 4);
    \fill[color = gray!50] (2, 2) rectangle (3, 3);
    \fill[color = gray!50] (3, 1) rectangle (4, 2);
    \fill[color = gray!50] (4, 0) rectangle (5, 1);
    \trigrid{5};
    \draw[color = red!80, thick] (1, 4) -- ++(-1, 0) -- ++(0, 1) -- ++ (2, 0) -- ++(0, -2) -- ++(-1, 0) -- cycle ;
    \draw[->, >=stealth, color = red!80, thick] (0.5, 4.5) -- (1.5, 4.5);
    \end{scope}
    
    \draw[->, >=stealth, thick] (5.5, -4.5)--(6.5, -4.5);
    
    \begin{scope}[xshift = 7cm, yshift = -7cm]
    \fill[color = gray!50] (1, 3) rectangle (2, 5);
    \fill[color = gray!50] (2, 2) rectangle (3, 3);
    \fill[color = gray!50] (3, 1) rectangle (4, 2);
    \fill[color = gray!50] (4, 0) rectangle (5, 1);
    \trigrid{5};
    \draw[color = red!80, thick] (2, 3) -- ++(-1, 0) -- ++(0, 1) -- ++ (2, 0) -- ++(0, -2) -- ++(-1, 0) -- cycle ;
    \draw[->, >=stealth, color = red!80, thick] (1.5, 3.5) -- (2.5, 3.5);
    \end{scope}
    
    \draw[->, >=stealth, thick, dashed] (12.5, -4.5)--(13.5, -4.5);
    
    \begin{scope}[xshift = 14cm, yshift = -7cm]
    \fill[color = gray!50] (2, 2) rectangle (3, 5);
    \fill[color = gray!50] (3, 1) rectangle (4, 2);
    \fill[color = gray!50] (4, 0) rectangle (5, 1);
    \trigrid{5};
    \draw[color = red!80, thick] (3, 2) -- ++(-1, 0) -- ++(0, 1) -- ++ (2, 0) -- ++(0, -2) -- ++(-1, 0) -- cycle ;
    \draw[->, >=stealth, color = red!80, thick] (2.5, 2.5) -- (3.5, 2.5);
    \end{scope}
    
    \draw[->, >=stealth, thick, dashed] (19.5, -4.5)--(20.5, -4.5);
    
    \begin{scope}[xshift = 21cm, yshift = -7cm]
    \fill[color = gray!50] (4, 0) rectangle (5, 5);
    \trigrid{5};
    \end{scope}
\end{tikzpicture}
    \caption{How to transform one line into another for $n=5$.}
    \label{fig:lines_transformation}
\end{figure}
\spacc{}
\qed\end{proof}

\subsection{The filling process}

The operation of \emph{filling} a pattern $P$ with the triangle shape was introduced in \cite{OEISfillmatrices} to define the fill matrices. It consists in recursively adding to $P$ the points that complete a triangle with two points of $P$ until no more can be added. More precisely, in a single step we may add $\vec w$ if $P \cap |\vec v + T| = 2$ and $\vec w \in \vec v + T$.

Clearly this transformation is confluent and terminating, and we denote by $\varphi(P)$ its unique the fixed point. Clearly $\varphi(P)$ upper bounds the set of points which can appear in patterns in the solitaire orbit of $P$. An example is shown in the appendix in Figure \ref{fig:fill process}.

We say $P$ \emph{fills} if $\varphi(P) = v + T_n$ for some vector $v$ where $n = |P|$. In \cite{OEISfillmatrices}, filling sets $P$ were called \emph{fill matrices}. We will show that they in fact correspond to the elements of the line orbit.

\begin{lemma}
\label{lem:fill_invar}
    If two patterns $P$ and $Q$ are in the same orbit, then $\varphi(P) = \varphi(Q)$.
\end{lemma}

\begin{lemma}
\label{lem:fill_shape}
    For any pattern $P$, there are unique integers $k_1, \ldots k_r$ and vectors $v_1, \ldots v_r$ such that $\varphi(P) = \bigcup_{i = 1}^r v_i + T_{k_i}$, $\sum_{i = 1}^r k_i \leqslant |P|$ and $N(T_{k_i}) \cap T_{k_j} = \varnothing$ for each $i \neq j$.
\end{lemma}

\begin{proof}
    We prove this by induction on $|P|$. The case $|P| = 1$ is trivial.
    
    Now assume the result is true for patterns of size at most $n$ and let $P$ be a pattern of size $n +1$. Then if $x \in P$, $P\setminus\{x\}$ satisfies the induction hypothesis so we can write $\varphi(P\setminus\{x\}) = \bigcup_{i= 1}^r v_i + T_{k_i}$. We now have three cases to consider. First, if $x \in \varphi(P\setminus\{x\})$ then $\varphi(P) = \varphi(P\setminus\{x\})$. Then, if $x$ is not in the neighbourhood of $\varphi(P\setminus\{x\})$ then no additional filling can be done with it, therefore $\varphi(P) = \varphi(P\setminus\{x\}) \cup \{x\}$ and, as $\{x\}$ is a triangle, we have the appropriate decomposition.
    
    Finally, assume $x \in N(v_1 + T_{k_1})$, then we can extend $T_{k_1} +v_1$ as in Figure \ref{fig:fill_triangle}. By doing so we may lose the property that the triangles do not touch, but if some do so we can merge them by repeating the extension process. Notice that if two triangles are merged, then the new triangle cannot be larger than the sum of the sizes of the initial triangle so the inequality on the triangles' sizes is still satisfied. (Merges may be triggered recursively, but nevertheless no merge can increase the sum of triangle sizes.)
    \begin{figure}[ht]
        \centering
        \begin{tikzpicture}
    \begin{scope}
    \fill[color = gray!50] (1, 4) rectangle (5, 5);
    \fill[color = gray!50] (2, 3) rectangle (5, 4);
    \fill[color = gray!50] (3, 2) rectangle (5, 3);
    \fill[color = gray!50] (4, 1) rectangle (5, 2);
    \fill[color = gray!70] (2, 2) rectangle (3, 3);
    \trigrid{5};
    \draw[color = red, very thick] (2, 3) -- ++(-1, 0) -- ++(0, 1) -- ++ (2, 0) -- ++(0, -2) -- ++(-1, 0) -- cycle ;
    \end{scope}
    
    \draw[->, >=stealth, thick] (5.5, 2.5)--(6.5, 2.5);
    
    \begin{scope}[xshift = 7cm]
    \fill[color = gray!50] (1, 4) rectangle (5, 5);
    \fill[color = gray!50] (2, 3) rectangle (5, 4);
    \fill[color = gray!50] (3, 2) rectangle (5, 3);
    \fill[color = gray!50] (4, 1) rectangle (5, 2);
    \fill[color = gray!70] (2, 2) rectangle (3, 3);
    \fill[color = gray!70] (1, 3) rectangle (2, 4);
    \trigrid{5};
    \draw[color = red, very thick] (3, 2) -- ++(-1, 0) -- ++(0, 1) -- ++ (2, 0) -- ++(0, -2) -- ++(-1, 0) -- cycle ;
    \draw[color = red, very thick] (1, 4) -- ++(-1, 0) -- ++(0, 1) -- ++ (2, 0) -- ++(0, -2) -- ++(-1, 0) -- cycle ;
    \end{scope}
    
    \draw[->, >=stealth, thick, dashed] (12.5, 2.5)--(13.5, 2.5);
    
    \begin{scope}[xshift = 14cm]
    \fill[color = gray!50] (1, 4) rectangle (5, 5);
    \fill[color = gray!50] (2, 3) rectangle (5, 4);
    \fill[color = gray!50] (3, 2) rectangle (5, 3);
    \fill[color = gray!50] (4, 1) rectangle (5, 2);
    \fill[color = gray!70] (2, 2) rectangle (3, 3);
    \fill[color = gray!70] (1, 3) rectangle (2, 4);
    \fill[color = gray!70] (0, 4) rectangle (1, 5);
    \fill[color = gray!70] (3, 1) rectangle (4, 2);
    \fill[color = gray!70] (4, 0) rectangle (5, 1);
    \trigrid{5};
    \end{scope}
    
    \begin{scope}[yshift = -7cm]
    \fill[color = gray!50] (1, 3) rectangle (5, 4);
    \fill[color = gray!50] (2, 2) rectangle (5, 3);
    \fill[color = gray!50] (3, 1) rectangle (5, 2);
    \fill[color = gray!50] (4, 0) rectangle (5, 1);
    \fill[color = gray!70] (2, 5) rectangle (3, 4);
    \trigrid{5};
    \draw[color = red, very thick] (3, 4) -- ++(-1, 0) -- ++(0, 1) -- ++ (2, 0) -- ++(0, -2) -- ++(-1, 0) -- cycle ;
    \end{scope}
    
    \draw[->, >=stealth, thick] (5.5, -4.5)--(6.5, -4.5);
    
    \begin{scope}[xshift = 7cm, yshift = -7 cm]
    \fill[color = gray!50] (1, 3) rectangle (5, 4);
    \fill[color = gray!50] (2, 2) rectangle (5, 3);
    \fill[color = gray!50] (3, 1) rectangle (5, 2);
    \fill[color = gray!50] (4, 0) rectangle (5, 1);
    \fill[color = gray!70] (2, 5) rectangle (4, 4);
    \trigrid{5};
    \draw[color = red, very thick] (4, 4) -- ++(-1, 0) -- ++(0, 1) -- ++ (2, 0) -- ++(0, -2) -- ++(-1, 0) -- cycle ;
    \draw[color = red, very thick] (2, 4) -- ++(-1, 0) -- ++(0, 1) -- ++ (2, 0) -- ++(0, -2) -- ++(-1, 0) -- cycle ;
    \end{scope}
    
    \draw[->, >=stealth, thick, dashed] (12.5, -4.5)--(13.5, -4.5);
    
    \begin{scope}[xshift = 14cm, yshift = -7cm]
    \fill[color = gray!50] (1, 3) rectangle (5, 4);
    \fill[color = gray!50] (2, 2) rectangle (5, 3);
    \fill[color = gray!50] (3, 1) rectangle (5, 2);
    \fill[color = gray!50] (4, 0) rectangle (5, 1);
    \fill[color = gray!70] (0, 5) rectangle (5, 4);
    \trigrid{5};
    \end{scope}
\end{tikzpicture}
        \caption{How to extend a triangle with a top neighbour or a subdiagonal neighbour. The right neighbour case is symmetric to the top neighbour case.}
        \label{fig:fill_triangle}
    \end{figure}
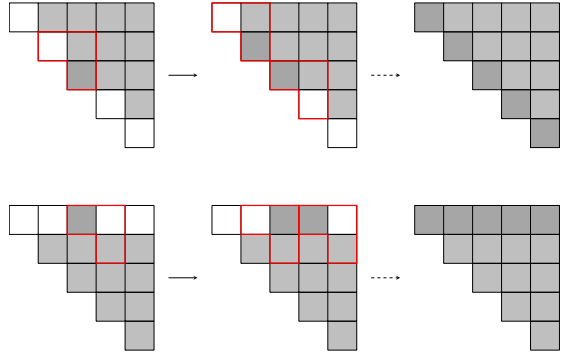
    \spacc{}
\qed\end{proof}

We define the \emph{excess} of $P$ as the difference $e(P) = |P| - \sum_{i=1}^r k_i$. Note that $P$ is a fill matrix if and only if it fills and has no excess. We note some basic properties of excess (proofs can be found in the appendix).

\begin{lemma}
\label{lem:monotony excess}
    If $Q$ is a subpattern of $P$ then $e(Q) \leqslant e(P)$.
\end{lemma}

\subsection{Characterisation of the orbit through the filling}

It what follow, two patterns \emph{touch} if their fillings touch, otherwise they are \emph{disjoint}.

\begin{theorem}
\label{thm:line_orbit}
    A pattern $P$ has no excess if and only if it is in the orbit of the lines that generate the $T_{k_i}$s.
\end{theorem}
\begin{proof}
    Lemmas \ref{lem:fill_shape} and \ref{lem:fill_invar} guarantee that if $P$ is in the orbit of a set of disjoint lines then it has no excess.
    
    We show the reverse implication by induction on $|P|$.
    
    The base case is trivial since $P$ can only be the line of length $1$.
    
    Now assume the implication holds for all patterns of size $n$ and let $|P|$ be a pattern of size $n+1$ without excess, and let $x \in P$. Then $P \setminus \{x\}$ still has no excess (by Lemma~\ref{lem:monotony excess})
    and has size $n$ so we can apply the induction hypothesis. Let $v_1 + T_{k_1}, \ldots, v_r + T_{k_r}$ be the partition of $\varphi(P\setminus\{x\})$ into triangles.
    
    We have $x \notin \bigcup_{i=1}^r v_i + T_{k_i}$ because otherwise $P$ would have excess. If $x$ is not in the union of the triangles' neighbourhoods then $\varphi(P) = \varphi(P\setminus\{x\}) \cup \{x$\} and $x$ cannot be moved by a triangle move. Else, $x$ can be moved so as to extend the line of the triangle it touches. Figure \ref{fig:merge_line} shows how to extend a line with a top neighbour, the two other cases are similar.
    
    In this process, two triangle might start touching, in this case, by repeating the previous operation we can merge the two triangles' lines into one. Notice that in every merge the length of the new line is exactly the sum of the length of the two initial lines (as otherwise excess is generated).
    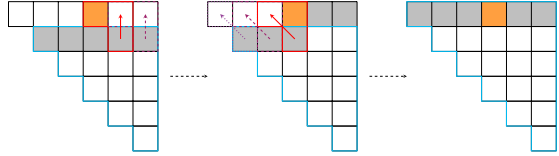
\begin{figure}[ht]
        \centering
        \begin{tikzpicture}
    \begin{scope}
    \fill[color = gray!50] (1, 4) rectangle (6, 5);
    \fill[color = orange!75] (3, 5) rectangle (4, 6);
    \trigrid{6};
    \draw[very thick, color = cyan!80] (1,5)--++(0, -1) -- ++(1, 0)--++(0, -1) -- ++(1, 0)--++(0, -1) -- ++(1, 0)--++(0, -1) -- ++(1, 0)--++(0, -1) -- ++(1, 0) -- (6, 5) --cycle;
    \draw[color = RedViolet, very thick, dashed] (5, 5) -- ++(-1, 0) -- ++(0, 1) -- ++ (2, 0) -- ++(0, -2) -- ++(-1, 0) -- cycle ;
    \draw[color = red, very thick] (4, 5) -- ++(-1, 0) -- ++(0, 1) -- ++ (2, 0) -- ++(0, -2) -- ++(-1, 0) -- cycle ;
    \draw[->, >=stealth, color = red, very thick] (4.5, 4.5) -- (4.5, 5.5);
    \draw[->, >=stealth, color = RedViolet, very thick, dashed] (5.5, 4.5) -- (5.5, 5.5);
    \end{scope}
    
    \draw[->, >=stealth, thick, dashed] (6.5, 3)--(8, 3);
    
    \begin{scope}[xshift = 8cm]
    \fill[color = gray!50] (1, 4) rectangle (4, 5);
    \fill[color = gray!50] (4, 5) rectangle (6, 6);
    \fill[color = orange!75] (3, 5) rectangle (4, 6);
    \trigrid{6};
    \draw[very thick, color = cyan!80] (1,5)--++(0, -1) -- ++(1, 0)--++(0, -1) -- ++(1, 0)--++(0, -1) -- ++(1, 0)--++(0, -1) -- ++(1, 0)--++(0, -1) -- ++(1, 0) -- (6, 5) --cycle;
    \draw[color = Purple, very thick, dotted] (1, 5) -- ++(-1, 0) -- ++(0, 1) -- ++ (2, 0) -- ++(0, -2) -- ++(-1, 0) -- cycle ;
    \draw[color = RedViolet, very thick, dashed] (2, 5) -- ++(-1, 0) -- ++(0, 1) -- ++ (2, 0) -- ++(0, -2) -- ++(-1, 0) -- cycle ;
    \draw[color = red, very thick] (3, 5) -- ++(-1, 0) -- ++(0, 1) -- ++ (2, 0) -- ++(0, -2) -- ++(-1, 0) -- cycle ;
    \draw[->, >=stealth, color = red, very thick] (3.5, 4.5) -- (2.5, 5.5);
    \draw[->, >=stealth, color = RedViolet, very thick, dashed] (2.5, 4.5) -- (1.5, 5.5);
    \draw[->, >=stealth, color = Purple, very thick, dotted] (1.5, 4.5) -- (0.5, 5.5);
    \end{scope}
    
    \draw[->, >=stealth, thick, dashed] (14.5, 3)--(16, 3);
    
    \begin{scope}[xshift = 16 cm]
    \fill[color = gray!50] (0, 5) rectangle (6, 6);
    \fill[color = orange!75] (3, 5) rectangle (4, 6);
    \trigrid{6};
    \draw[very thick, color = cyan!80] (0,6)--++(0, -1) -- ++(1, 0)--++(0, -1) -- ++(1, 0)--++(0, -1) -- ++(1, 0)--++(0, -1) -- ++(1, 0) --++(0, -1) -- ++(1, 0)--++(0, -1) -- ++(1, 0) -- (6, 6) --cycle;
    \end{scope}
\end{tikzpicture}
        \caption{How to extend a line with a top neighbour.}
        \label{fig:merge_line}
    \end{figure}
    \spacc{}
\qed\end{proof}

We denote by $P_{n, k}$ the shape composed of a line of length $n$ to which $k$ points were added by filling the triangle under the line from right to left and top to bottom. Examples are shown Figure \ref{fig:Pnk}. 
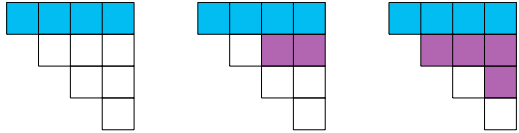
\begin{figure}[ht]
    \centering
    \begin{tikzpicture}
    \fill[color = cyan!75] (0,3) rectangle (4, 4); 
    \trigrid{4}

    \begin{scope}[xshift = 6cm]
    \fill[color = cyan!75] (0,3) rectangle (4, 4); 
    \fill[color = violet!60] (2,2) rectangle (4, 3); 
    \trigrid{4}
    \end{scope}
    
    \begin{scope}[xshift = 12cm]
    \fill[color = cyan!75] (0,3) rectangle (4, 4); 
    \fill[color = violet!60] (1,2) rectangle (4, 3);
    \fill[color = violet!60] (3,1) rectangle (4, 2);
    \trigrid{4}
    \end{scope}
\end{tikzpicture}
    \caption{From left to right: $P_{4, 0}$, $P_{4, 2}$ and $P_{4, 4}$. The purple cells are the excess.}
    \label{fig:Pnk}
\end{figure}
\spacc{}

\begin{theorem}
\label{thm:excess_orbit}
    If $P$ is a pattern, then $P \in \gamma(P_{n, k})$ if and only if $\varphi(P) = T_n$ and $e(P) = k$.
\end{theorem}
\begin{proof}
    Triangle moves preserve fillings and excess so the direct implication holds.
    
    First notice that any excess forming lines under an already existing full top line can be normalized to the corresponding $P_{n, k}$ using the transformations shown in Figure \ref{fig:push excess}.
    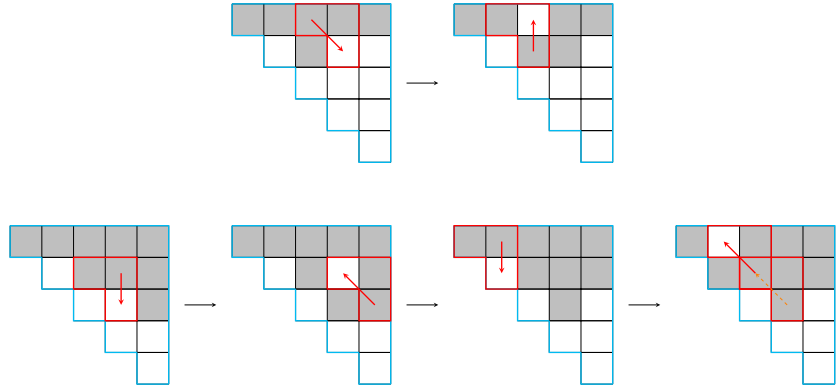
\begin{figure}[ht]
        \centering
        \begin{tikzpicture}
    \begin{scope}[xshift = 7cm]
    \fill[color = gray!50] (0,4) rectangle (5,5);
    \fill[color = gray!50] (2,3) rectangle (3,4);
    \trigrid{5};
    \TriangShape{5}{very thick, color = cyan!80};
    \Lshape{color=red, very thick}{(3,4)};
    \move{red}{(2.5, 4.5)}{(3.5, 3.5)};
    \end{scope}
    
    \step{}{(12.5, 2.5)};
    
    \begin{scope}[xshift = 14cm]
    \fill[color = gray!50] (0,4) rectangle (2,5);
    \fill[color = gray!50] (2,3) rectangle (4,4);
    \fill[color = gray!50] (3,4) rectangle (5,5);
    \trigrid{5};
    \TriangShape{5}{very thick, color = cyan!80};
    \Lshape{color=red, very thick}{(2,4)};
    \move{red}{(2.5, 3.5)}{(2.5, 4.5)};
    \end{scope}

    \begin{scope}[yshift = -7cm]
    \fill[color = gray!50] (0,4) rectangle (5,5);
    \fill[color = gray!50] (2,3) rectangle (5,4);
    \fill[color = gray!50] (4,2) rectangle (5,3);
    \trigrid{5};
    \TriangShape{5}{very thick, color = cyan!80};
    \Lshape{color=red, very thick}{(3,3)};
    \move{red}{(3.5, 3.5)}{(3.5, 2.5)};
    \end{scope}
    
    \step{}{(5.5, -4.5)};
    
    \begin{scope}[xshift = 7cm, yshift = -7cm]
    \fill[color = gray!50] (0,4) rectangle (5,5);
    \fill[color = gray!50] (2,3) rectangle (3,4);
    \fill[color = gray!50] (4,3) rectangle (5,4);
    \fill[color = gray!50] (3,2) rectangle (5,3);
    \trigrid{5};
    \TriangShape{5}{very thick, color = cyan!80};
    \Lshape{color=red, very thick}{(4,3)};
    \move{red}{(4.5, 2.5)}{(3.5, 3.5)};
    \end{scope}
    
    \step{}{(12.5, -4.5)};
    
    \begin{scope}[xshift = 14cm, yshift = -7cm]
    \fill[color = gray!50] (0,4) rectangle (5,5);
    \fill[color = gray!50] (2,3) rectangle (5,4);
    \fill[color = gray!50] (3,2) rectangle (4,3);
    \trigrid{5};
    \TriangShape{5}{very thick, color = cyan!80};
    \Lshape{color=red, very thick}{(1,4)};
    \move{red}{(1.5, 4.5)}{(1.5, 3.5)};
    \end{scope}
    
    \step{}{(19.5, -4.5)};
    
    \begin{scope}[xshift = 21cm, yshift = -7cm]
    \fill[color = gray!50] (0,4) rectangle (1,5);
    \fill[color = gray!50] (2,4) rectangle (5,5);
    \fill[color = gray!50] (1,3) rectangle (5,4);
    \fill[color = gray!50] (3,2) rectangle (4,3);
    \trigrid{5};
    \TriangShape{5}{very thick, color = cyan!80};
    \Lshape{color=red, very thick}{(2,4)};
    \move{red}{(2.5, 3.5)}{(1.5, 4.5)};
    \Lshape{color=red, very thick}{(3,3)};
    \draw[->, >=stealth, color = orange,  thick, dashed] (3.5, 2.5)--(2.5, 3.5);
    \end{scope}
\end{tikzpicture}
        \caption{Pushing the excess to the right (first line) or to the left (second line).}
        \label{fig:push excess}
    \end{figure}
    \spacc{}
    
    Then, if the pattern is composed of a line and one additional point, the additional point can be fetched by transforming the line into the diagonal, stopping the process when one point is above the additional point, then getting back to the line, while dragging the additional point.
    
    All that remains to do is to push the excess to the right. Figure \ref{fig:fetch excess} provides an example. Notice that the presence of some excess already in a good position will not cause any problem for this method, as illustrated Figure~\ref{fig:fetch more excess}.
    \begin{figure}[ht]
        \centering
        \begin{tikzpicture}
    \begin{scope}
    \fill[color = gray!50] (0,5) rectangle (6,6);
    \fill[color = gray!50] (3,2) rectangle (4,3);
    \trigrid{6};
    \TriangShape{6}{very thick, color = cyan!80};
    \Lshape{color=orange, very thick}{(5,5)};
    \move{orange}{(5.5, 5.5)}{(5.5, 4.5)};
    \Lshape{color = red, thick, dashed}{(4,5)};
    \draw[color = red, thick, ->, >=stealth, dashed] (4.5, 5.5) --++(0,-1);
    \Lshape{color = purple, thick, dotted}{(3,5)};
    \draw[color = purple, ->, >=stealth, dotted] (3.5, 5.5) --++(0,-1);
    \end{scope}
    
    \step{dashed}{(6.5, 3)};
    
    \begin{scope}[xshift = 8cm]
    \fill[color = gray!50] (0,5) rectangle (1,6);
    \fill[color = gray!50] (1,4) rectangle (3,5);
    \fill[color = gray!50] (3,3) rectangle (6,4);
    \fill[color = gray!50] (3,2) rectangle (4,3);
    \trigrid{6};
    \TriangShape{6}{very thick, color = cyan!80};
    \Lshape{color=red, very thick}{(3,3)};
    \move{red}{(3.5, 2.5)}{(2.5, 3.5)};
    \end{scope}
    
    \step{}{(14.5, 3)};
    
    \begin{scope}[xshift = 16cm]
    \fill[color = gray!50] (0,5) rectangle (1,6);
    \fill[color = gray!50] (1,4) rectangle (3,5);
    \fill[color = gray!50] (2,3) rectangle (6,4);
    \trigrid{6};
    \TriangShape{6}{very thick, color = cyan!80};
    \Lshape{color=orange, very thick}{(3,4)};
    \move{orange}{(3.5, 3.5)}{(3.5, 4.5)};
    \Lshape{color = red, thick, dashed}{(4,4)};
    \draw[color = red, thick, ->, >=stealth, dashed] (4.5, 3.5) --++(0,1);
    \Lshape{color = purple, thick, dotted}{(5,4)};
    \draw[color = purple, ->, >=stealth, dotted] (5.5, 3.5) --++(0,1);
    \end{scope}
    
    \step{dashed}{(-0.5, -5)};
    
    \begin{scope}[xshift = 1cm, yshift = -8 cm] 
    \fill[color = gray!50] (0,5) rectangle (1,6);
    \fill[color = gray!50] (1,4) rectangle (6,5);
    \fill[color = gray!50] (2,3) rectangle (3,4);
    \trigrid{6};
    \TriangShape{6}{very thick, color = cyan!80};
    \Lshape{color=orange, very thick}{(1,5)};
    \move{orange}{(1.5, 4.5)}{(1.5, 5.5)};
    \Lshape{color = red, thick, dashed}{(2,4)};
    \draw[color = red, thick, ->, >=stealth, dashed] (2.5, 3.5) --++(-1,1);
    \end{scope}
    
    \step{dashed}{(7.5, -5)};
    
    \begin{scope}[xshift = 9cm, yshift = -8cm]
    \fill[color = gray!50] (0,5) rectangle (2,6);
    \fill[color = gray!50] (1,4) rectangle (6,5);
    \trigrid{6};
    \TriangShape{6}{very thick, color = cyan!80};
    \Lshape{color=orange, very thick}{(2,5)};
    \move{orange}{(2.5, 4.5)}{(2.5, 5.5)};
    \Lshape{color = red, thick, dashed}{(3,5)};
    \draw[color = red, thick, ->, >=stealth, dashed] (3.5, 4.5) --++(0,1);
    \Lshape{color = purple, thick, dotted}{(3,5)};
    \draw[color = purple, ->, >=stealth, dotted] (4.5, 4.5) --++(0,1);
    \end{scope}
    
    \step{dashed}{(15.5, -5)};
    
    \begin{scope}[xshift = 17cm, yshift = -8cm]
    \fill[color = gray!50] (0,5) rectangle (6,6);
    \fill[color = gray!50] (1,4) rectangle (2,5);
    \trigrid{6};
    \TriangShape{6}{very thick, color = cyan!80};
    \end{scope}
\end{tikzpicture}
        \caption{Fetching an excess point.}
        \label{fig:fetch excess}
    \end{figure}
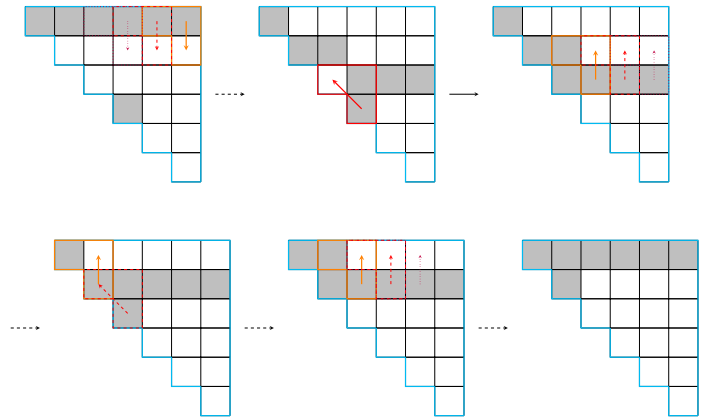
    \spacc{}
    \begin{figure}[ht]
        \centering
        \begin{tikzpicture}
    \begin{scope}
    \fill[color = gray!50] (0,4) rectangle (5,5);
    \fill[color = gray!50] (3,3) rectangle (5,4);
    \fill[color = gray!50] (3,1) rectangle (4,2);
    \trigrid{5};
    \TriangShape{5}{very thick, color = cyan!80};
    \Lshape{color=orange, very thick}{(2,4)};
    \move{orange}{(2.5, 4.5)}{(2.5, 3.5)};
    \Lshape{color = red, thick, dashed}{(1,4)};
    \draw[color = red, thick, ->, >=stealth, dashed] (1.5, 4.5) --++(0,-1);
    \end{scope}
    
    \step{dashed}{(5.5, 2.5)};
    
    \begin{scope}[xshift = 7cm]
    \fill[color = gray!20] (0,4) rectangle (1,5);
    \fill[color = gray!50] (1,3) rectangle (5,4);
    \fill[color = gray!20] (3,4) rectangle (5,5);
    \fill[color = gray!50] (3,1) rectangle (4,2);
    \trigrid{5};
    \begin{scope}[xshift = 1cm]\TriangShape{4}{very thick, color = cyan!80}\end{scope};
    \end{scope}
    
    \step{dashed}{(12.5, 2.5)};
    
    \begin{scope}[xshift = 14cm]
    \fill[color = gray!20] (0,4) rectangle (1,5);
    \fill[color = gray!20] (3,4) rectangle (5,5);
    \fill[color = gray!50] (1,3) rectangle (3,4);
    \fill[color = gray!50] (3,2) rectangle (5,3);
    \fill[color = gray!50] (3,1) rectangle (4,2);
    \trigrid{5};
    \begin{scope}[xshift = 1cm]\TriangShape{4}{very thick, color = cyan!80}\end{scope};
    \Lshape{color = orange, very thick}{(3,2)};
    \move{orange}{(3.5, 1.5)}{(2.5, 2.5)};
    \Lshape{color = red, thick, dashed}{(3,3)};
    \draw[color = red, thick, ->, >=stealth, dashed] (3.5, 2.5) --++(0,1);
    \Lshape{color = purple, thick, dotted}{(4,3)};
    \draw[color = purple, thick, ->, >=stealth, dotted] (4.5, 2.5) --++(0,1);
    \end{scope}
    
    \step{dashed}{(19.5, 2.5)};
    
    \begin{scope}[xshift = 21cm]
    \fill[color = gray!50] (0,4) rectangle (1,5);
    \fill[color = gray!50] (3,4) rectangle (5,5);
    \fill[color = gray!50] (1,3) rectangle (5,4);
    \fill[color = gray!50] (2,2) rectangle (3,3);
    \trigrid{5};
    \TriangShape{5}{very thick, color = cyan!80};
    \Lshape{color=orange, very thick}{(1,4)};
    \move{orange}{(1.5, 3.5)}{(1.5, 4.5)};
    \Lshape{color = red, thick, dashed}{(2,4)};
    \draw[color = red, thick, ->, >=stealth, dashed] (2.5, 3.5) --++(0,1);
    \Lshape{color = purple, thick, dotted}{(2,3)};
    \draw[color = purple, ->, >=stealth, dotted] (2.5, 2.5) --++(-1,1);
    \end{scope}
\end{tikzpicture}
        \caption{Fetching an excess point with some excess already lined up.}
        \label{fig:fetch more excess}
    \end{figure}
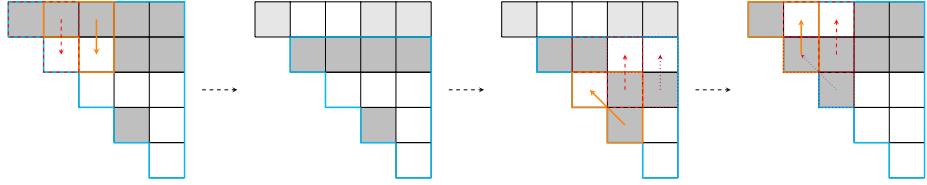
    
    Thus, any pattern composed of a line of length $n$ and $k$ excess points under it is in the orbit of $P_{n,k}$.
    
    All that remains to prove is that a line can always be formed. This can be done using the merging process described in the proof of Theorem \ref{thm:line_orbit}. The excess point will not disturb the merging because if one move is made impossible by their presence, say point $x$ cannot be moved to position $y$, then we can just use $y$ instead of $x$ for the following steps. \qed
\end{proof}

Combining Theorem \ref{thm:line_orbit} and Theorem \ref{thm:excess_orbit}, we obtain the following classification of the orbits.

\begin{theorem}[Characterisation of the orbits]
\label{thm:orbits_charac}
    If $P$ is a finite pattern then there are integers $n_1, \ldots n_r$ and $k_1, \ldots k_r$ and vectors $v_1, \ldots v_r$ such that $P \in \gamma(\bigcup_{i = 1}^r v_i + P_{n_i, k_i})$, the $P_{n_i, k_i} + v_i$ do not touch each other, $\sum_{i=1}^r n_i = |P|-e(P)$ and $\sum_{i=1}^r k_i =e(P)$.
\end{theorem}

Note that we now know the orbits for solitaires process with all shapes of size 3. Indeed, such a shape is either 3 aligned points (in which case orbits are easy to analyse), 
or it is a triangle shape on a finite index subgroup of $\mathbb{Z}^2$, and the orbits in different cosets of this subgroup are completely independent and are individually described by the triangle solitaire.

\subsection{Excess sets}
\label{sec:excess sets}

It is tempting to think that if a set has excess, then we can remove some of its points to remove the excess. This turns out not to be true.

If $P$ is a pattern, the \emph{excess sets} of $P$ are the subsets $Q \subset P$ such that $\varphi(P\setminus Q) = \varphi(P)$. We denote by $E(P)$ the set of all such sets.

\begin{lemma}
    If $U \in E(P)$ then $|U| \leqslant e(P)$.
\end{lemma}
\begin{proof}
    This is a direct consequence of Lemma \ref{lem:fill_shape}.
\end{proof}

However, in general there is no set $U \in E(P)$ such that $|U| = e(P)$. The pattern in Figure~\ref{fig:no excess point} has excess 1 but its only excess set is the empty set.

\begin{figure}[ht]
    \centering
    \begin{tikzpicture}
    \fill[color = gray!50] (0,6) rectangle (1, 7);
    \fill[color = gray!50] (1,5) rectangle (2, 6);
    \fill[color = gray!50] (3,5) rectangle (4, 7);
    \fill[color = gray!50] (5,3) rectangle (7, 4);
    \fill[color = gray!50] (5,1) rectangle (6, 2);
    \fill[color = gray!50] (6,0) rectangle (7, 1);
    \trigrid{7}
\end{tikzpicture}
    \caption{Here, $e(p) = 1$ but $E(P) = \{ \varnothing \}$.}
    \label{fig:no excess point}
\end{figure}

If $U \in E(P)$, then every $V \subset U$ is also an excess set. Maximal excess sets do not all have the same cardinality as shown by example in Figure~\ref{fig:max excess set}.

\begin{figure}[ht]
    \centering
    \begin{tikzpicture}
    \fill[color = gray!50] (0,6) rectangle (1, 7);
    \fill[color = gray!50] (1,5) rectangle (2, 6);
    \fill[color = cyan!70] (3,5) rectangle (4, 7);
    \fill[color = orange!70] (4,4) rectangle (5, 5);
    \fill[color = gray!50] (5,3) rectangle (7, 4);
    \fill[color = gray!50] (5,1) rectangle (6, 2);
    \fill[color = gray!50] (6,0) rectangle (7, 1);
    \trigrid{7}
\end{tikzpicture}
    \caption{The blue and orange sets are both maximal excess sets but do not have the same cardinality.}
    \label{fig:max excess set}
\end{figure}
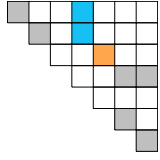

\section{Algorithmic aspects}

\subsection{Complexity of the identification of the orbit of a pattern}

The proof of the characterisation of the orbits provides a polynomial time algorithm to identify to which orbit a given pattern belongs. 

The algorithm is the following:
\begin{alg}[Identify orbit]
\label{algo:ident orbit}
Data: pattern $P$. Result: the canonical representative of the orbit of $P$.
\begin{enumerate}
    \item Fill the pattern.
    \item Divide the filling into triangles $v_1 + T_{k_1}, \ldots, v_r + T_{k_r}$.
    \item Count the excess in each triangle, the canonical representative of the orbit of the pattern is $\bigcup_{i = 1}^r v_i + P_{k_i, e(P \cap (v_i+T_{k_i}))}$. 
\end{enumerate}
\end{alg}


Let $n = |P|$. The first two steps are linear in the number of points in $\varphi(P)$ and $|\varphi(P)| \leqslant \frac{n(n+1)}2$ so they run in time $O(n^2)$. Step 3 is then linear so the total time complexity of the algorithm is $O(n^2)$.

\subsection{Number of steps needed to put a pattern in normal form}

The proof of the characterisation of the orbit also provides an algorithm for finding a sequence of transformations turning the pattern into the canonical representative of its orbit, following the proof of Theorem~\ref{thm:orbits_charac}:
\begin{alg}
\label{alg:path}
\begin{enumerate}
    \item Merge the different components and form lines using the process described in Theorem \ref{thm:line_orbit} (accounting for excess using the modifications described in the proof of Theorem~\ref{thm:excess_orbit}).
    \item Fetch the excess with the process described in Theorem~\ref{thm:excess_orbit}.
\end{enumerate}
\end{alg}

Naively implemented, this algorithm takes $O(n^2 (n + k))$, where $k$ denotes the excess: The first step takes $O(n^3)$, since each merging takes $O(n^2)$ time and we merge at most $n$ times. The second step takes $O(n^2 k)$ if we fetch the $k$ many excess points one by one, as each fetch takes $O(n^2)$.

For $k = 0$, this is in fact optimal.

\begin{theorem}
\label{thm:diameter}
    The diameter of the orbit of the line of length $n$, seen as a graph, is $\Theta(n^3)$.
\end{theorem}
\begin{proof}
    We are going to build an infinite family of patterns that require $\Omega(n^3)$ steps to get back to the line.
    
    Let $P_0$ be the empty pattern, and $P_{n+1}$ is inductively built by extending $P_n$ as described in Figure \ref{fig:far pattern} where the grey triangle is the triangle in which  $P_n$'s orbit is confined.
    \begin{figure}[ht]
        \centering
        \includegraphics[height = 2.5cm]{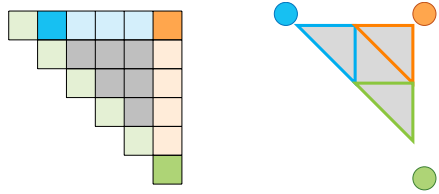}
        \caption{Left: The extension of $P_n$ into $P_{n+1}$. Right: A schematic representation of $P_{n+1}$ used in the proof.}
        \label{fig:far pattern}
    \end{figure}
    \spacc{}
    In pattern $P_{n+1}$, $\Omega(n^2)$ steps are required to fetch the three coloured points. 
    
    Indeed, first notice that up to renaming points, the blue point has to move right to fetch the orange one, then the orange one will have to move down to fetch the green one and finally the green one will have to go up to prepare for the next extension. Now let us analyse the movement of the blue point, starting from the first moment it touches another triangle. To move a point right, one need a point in the column at its right. We'll prove that this means that at some point in the process, half of the points of the pattern have to be in the left half of the triangle.
    
    Mark the blue point, and whenever an unmarked point is in the same column as a marked one, mark it. Consider column $i$ from the left, and the first marked point $x$ to reach it. Clearly the point $x$ moves from column $i-1$ to column $i$, and this requires us to have an unmarked point in column $i$ to allow this. This point was unmarked, and is now marked. Therefore, when the blue point has reached column $\frac{3n}2$ (the middle column), there are at least $\frac{3n}2$ marked points, all of which were in the blue triangle in Figure~\ref{fig:far pattern} at some point during the journey of the blue point.
    
    The same reasoning on the orange and green point gives that at some point of the process, between the first moment the blue point moves, until the point where we are ready to move the blue point of the next level, $\frac{3n}2$ points were in the orange triangle and the same amount in the green one. As the pattern $P_n$ only has $3n$ points, at least $\frac{3n}2$ need to be moved from one subtriangle to another, therefore there is a pair of triangles that will share at least $\frac{n}2$ points. Those at least $\frac{n}2$ points will need to be moved by a mean distance of at least $\frac{n}4$, which requires at least $\frac{n^2}8$ steps. Thus pattern $P_n$ requires $\Omega(n^3)$ moves to be transformed into a line.
\qed\end{proof}

\section{Size of the Line Orbit}

One can build an element of the orbit of the line by choosing a corner, then choosing a point on each line parallel to the edge opposed to the corner. This gives $3n! -3$ elements of the orbit as the only patterns that can be created this way from two different corners are the lines.

If $P$ is in the line orbit, then the number of points in $P$ in the first $k$ columns is at most $k$, and \cite{OEISupperbound} 
gives the count 
$c\left(\frac{e}2\right)^n(n-1)^{n-\frac52}$ with $c = \frac{4 + 2\text{LambertW}(-2e^{-2})}{e^3\sqrt{2\pi}}$ for patterns with this property.

Combining these, we get the following bounds on the size of the line orbit.

\begin{theorem}
\label{thm:orbit size}
    There are constants $c_1$ and $c_2$ such that $c_1e^{-n}n^{n + \frac12} \leqslant |\gamma(L_n)| \leqslant c_2 \left(\frac{e}2\right)^n(n-1)^{n-\frac52}$.
\end{theorem}

\begin{conjecture}
    There is a constant $\frac2e \leqslant c \leqslant e$ such that $|\gamma(L_n)| = \Theta \left(\left(\frac{n}c\right)^n\right)$.
\end{conjecture}

\section{Bases}

Let $X_R$ be a TEP subshift for some $R \subset \Sigma^T$ and fix some $n \in \N$. Recall from \cite{Sa22} that the pattern $q \in \Sigma^{T_n}$ appears in $X_R$ if and only if it does not explicitly contain a translate of a pattern in $\Sigma^T \setminus R$.

Let $P \subset T_n$. We say $P$ is a \emph{basis} (for the restriction of $X_R$ to $T_n$) if for every $q \in \Sigma^P$, by iteratively completing partially filled triangles using the TEP rule $R$, we always end up with a valid pattern on $T_n$.

\begin{theorem}
\label{thm:Bases}
The following are equivalent for $P \subset T_n$:
\begin{enumerate}
    \item $P$ is a basis,
    \item $P$ is a fill matrix,
    \item $P$ is in the line orbit.
\end{enumerate}
\end{theorem}

Note that the first item talks about a specific (but arbitrary) TEP subshift, while the rest do not, i.e.\ the set of bases is independent of $\Sigma$ and $R$.

The solitaire process allows us to translate patterns on one basis to ones on another, more space efficiently than the direct method suggests, indeed we give a polynomial time in-place algorithm for this.

If $P, Q \subset T_n$ are bases, any pattern $p \in \Sigma^P$ uniquely determines a pattern in $q \in \Sigma^Q$ in an obvious way, by deducing the unique extension of one pattern to $T_n$ and then restricting to the domain of the other. If we biject $P$ and $Q$ with $\{1, \ldots, n\}$ we obtain a bijection $\phi : \Sigma^n \to \Sigma^n$. A \emph{basic permutation} of $\Sigma^n$ is one that ignores all but two cells. If the cells are $1, 2$, this means that for some $\hat\pi \in \mathrm{Sym}(\Sigma^2)$ we have
$\pi(a_1 a_2 a_3 \cdots a_n)_i = \hat\pi(a_1 a_2) a_3 \cdots a_n$
for all $a_1a_2\cdots a_n \in \Sigma^n$; in general one conjugates by a reordering the cells.

\begin{theorem}
\label{thm:Memoryless}
The bijection $\phi$ can be computed with $O(n^3)$ basic permutations.
\end{theorem}

\section{Prospects for future work}

Arguments similar to those of this paper work for several other shapes, but for now we have not found a general result, in particular Question~5.36 from \cite{Sa22} stays open.

While Theorem~\ref{thm:Bases} characterises the patterns that generate the contents of a triangle, it does not characterise the maximal (or maximum cardinality) independent sets, i.e.\ patterns $P \subset T_n$ such that $X|_P = \Sigma^P$, where $X$ is a TEP subshift. The characterization of such sets seems more difficult and does depend on the specific TEP subshift.

We introduced in Section~\ref{sec:excess sets} the notion of an excess set. What can be said about the family of excess sets as a set system? How can we determine the maximum cardinality of an excess set?


%
%
%
 \bibliographystyle{splncs04}
 \bibliography{biblio} 

\newpage
\appendix

\section*{Omitted proofs}

\begin{proof}[Proof of Lemma~\ref{lem:monotony excess}]
    Denote $Q = \{x_1, \ldots x_m\}$ and $P = \{x_1, \ldots, x_n\}$. Then build the fillings of the $P_j = \{x_1, \ldots x_j\}$. With the proof of the previous lemma and using the same notations, we obtain that for each $1 \leqslant j \leqslant n$, either $x_j \in \varphi(P_{j-1})$ and $e(P_{j}) > e(P_{j-1})$ or adding $x_j$ augments $\sum_{i=1}^{r_{j-1}} k_{i, j-1}$ by one and $e(P_{j}) = e(P_{j-1})$ if no merges are triggered. Merges may only increase the amount of excess, as they can only decrease the sum of triangle sizes. Thus $e$ is monotonous.
\qed\end{proof}

\begin{proof}[Proof of Lemma~\ref{lem:fill_invar}]
    There is a sequence of transformation $t_1, \ldots, t_k$ that transforms $P$ into $Q$. Starting with $P$, we can fill the triangles at position $t_1, \ldots, t_k$, thus adding all the points of $Q$ to $P$, so $\varphi(Q) \subset \varphi(P)$. The other inclusion is symmetric.
\qed\end{proof}

\begin{proof}[Proof of Theorem~\ref{thm:Bases}]
The equivalence of the last two items is a special case of Theorem~\ref{thm:line_orbit}. We show that (1) implies (2) and (3) implies (1). 

Suppose thus first that $P$ is a basis. Then by definition $\varphi(P) = T_n$, as the iterative steps of completing triangles correspond exactly to the basic steps of the filling process. We cannot have more than $n$ elements in $P$, as the total number of valid patterns on $\Sigma^{T_n}$ is $n$. By Lemma~~\ref{lem:fill_shape} we also cannot have less, or the set of eventually filled positions could not possibly cover $T_n$. We conclude that $P$ is a fill matrix.

Suppose next that $P$ is in the line orbit. Let $P_0 = P, P_1, P_2, \ldots, P_k$ be a sequence of successive steps in the solitaire such that $P_k$ is an edge of $T_n$. It is easy to see that the property of being a basis is preserved under solitaire moves, indeed each solitaire step induces a natural bijection $\phi : \Sigma^{P_i} \to \Sigma^{P_{i+1}}$ and the unique completions of $p$ and $\phi(p)$ are the same. Since the line is clearly a basis, $P$ must be a basis as well.
\qed\end{proof}

\begin{proof}[Proof of Theorem~\ref{thm:Memoryless}]
Let $P_0 = P, P_1, \ldots, P_k = Q$ be the sequence of solitaire moves, which we know can be taken of length $O(n^3)$. Keep track of an ordering of the cells in $P_i$ so that if a solitaire move touches cells $P_i$ with indices $j_1, j_2$, then only the vectors with indices $j_1, j_2$ differ between $P_i$ and $P_{i+1}$. Then at each solitaire step, the TEP family allows us to compute the unique pattern in $\Sigma^{P_{i+1}}$ compatible with $\Sigma^{P_i}$ with a single basic permutation. After $k$ steps we obtain the pattern $\Sigma^Q$, but in a random order, and we can fix the ordering by applying $O(n)$ basic permutations with $\hat\pi$ of the form $ab \leftrightarrow ba$.
\qed\end{proof}

\newpage
\section*{Omitted Figures}

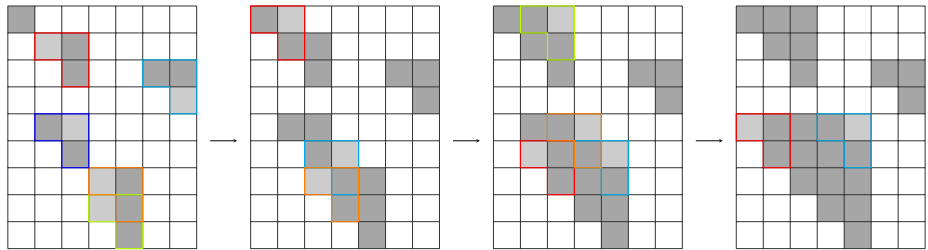
\begin{figure}[ht]
    \centering
    \begin{tikzpicture}
    \begin{scope}
    \fill[color = gray!70] (4,0) rectangle (5, 3);
    \fill[color = gray!70] (2,3) rectangle (3, 4);
    \fill[color = gray!70] (1,4) rectangle (2, 5);
    \fill[color = gray!70] (6,6) rectangle (7, 7);
    \fill[color = gray!70] (2,6) rectangle (3, 8);
    \fill[color = gray!70] (0,8) rectangle (1, 9);
    \fill[color = gray!70] (5,6) rectangle (7, 7);
    \fill[color = gray!40] (2,5) rectangle (3, 4);
    \fill[color = gray!40] (3,1) rectangle (4, 3);
    \fill[color = gray!40] (1,7) rectangle (2, 8);
    \fill[color = gray!40] (6,5) rectangle (7, 6);
    \draw[color=black!80] (0, 0) grid (7,9);
    \Lshape{color=red, very thick}{(2,7)};
    \Lshape{color=lime, very thick}{(4,1)};
    \Lshape{color=orange, very thick}{(4,2)};
    \Lshape{color=blue, very thick}{(2,4)};
    \Lshape{color=cyan, very thick}{(6,6)};
    \end{scope}
    
    \step{}{(7.5, 4)};
    
    \begin{scope}[xshift = 9cm]
    \fill[color = gray!40] (2,2) rectangle (4, 4);
    \fill[color = gray!70] (4,0) rectangle (5, 3);
    \fill[color = gray!70] (2,3) rectangle (3, 5);
    \fill[color = gray!70] (1,4) rectangle (2, 5);
    \fill[color = gray!70] (6,5) rectangle (7, 7);
    \fill[color = gray!70] (2,6) rectangle (3, 8);
    \fill[color = gray!70] (0,8) rectangle (1, 9);
    \fill[color = gray!70] (5,6) rectangle (6, 7);
    \fill[color = gray!70] (3,1) rectangle (4, 3);
    \fill[color = gray!70] (1,7) rectangle (2, 8);
    \fill[color = gray!40] (1,8) rectangle (2, 9);
    \draw[color = black!80] (0, 0) grid (7,9);
    \Lshape{color=red, very thick}{(1,8)};
    \Lshape{color=cyan, very thick}{(3,3)};
    \Lshape{color=orange, very thick}{(3,2)};
    \end{scope}
    
    \step{}{(16.5, 4)};
    
    \begin{scope}[xshift = 18cm]
    \fill[color = gray!40] (1,3) rectangle (5,4);
    \fill[color = gray!70] (2,2) rectangle (4,4);
    \fill[color = gray!70] (4,0) rectangle (5,3);
    \fill[color = gray!70] (1,4) rectangle (3,5);
    \fill[color = gray!70] (6,5) rectangle (7,7);
    \fill[color = gray!70] (2,6) rectangle(3,8);
    \fill[color = gray!70] (0,8) rectangle (3,9);
    \fill[color = gray!70] (5,6) rectangle (6,7);
    \fill[color = gray!70] (3,1) rectangle (4,2);
    \fill[color = gray!70] (1,7) rectangle (2,8);
    \fill[color = gray!40] (3,4) rectangle (4,5);
    \fill[color = gray!40] (2,8) rectangle (3, 9);
    \draw[color = black!80] (0, 0) grid (7,9);
    \Lshape{color=lime, very thick}{(2,8)};
    \Lshape{color=red, very thick}{(2,3)};
    \Lshape{color=cyan, very thick}{(4,3)};
    \Lshape{color=orange, very thick}{(3,4)};
    \end{scope}
    
    \step{}{(25.5, 4)};
    
    \begin{scope}[xshift = 27cm]
    \fill[color = gray!40] (0,4) rectangle (5, 5);
    \fill[color = gray!70] (1,3) rectangle (5, 4);
    \fill[color = gray!70] (2,2) rectangle (4, 4);
    \fill[color = gray!70] (4,0) rectangle (5, 3);
    \fill[color = gray!70] (1,4) rectangle (4, 5);
    \fill[color = gray!70] (3,1) rectangle (4, 2);
    \fill[color = gray!70] (6,5) rectangle (7,7);
    \fill[color = gray!70] (2,6) rectangle(3,8);
    \fill[color = gray!70] (0,8) rectangle (3,9);
    \fill[color = gray!70] (5,6) rectangle (6,7);
    \fill[color = gray!70] (1,7) rectangle (2,8);
    \draw[color = black!80] (0, 0) grid (7,9);
    \Lshape{color=red, very thick}{(1,4)};
    \Lshape{color=cyan, very thick}{(4,4)};
    \end{scope}
\end{tikzpicture}
    \caption{An example of filling process.}
    \label{fig:fill process}
\end{figure}

\begin{figure}[ht]
    \centering
    \begin{tikzpicture}
    \begin{scope}
    \fill[color = gray!50] (1,11) rectangle (3,12);
    \fill[color = gray!50] (5,11) rectangle (7,12);
    \fill[color = gray!50] (2,9) rectangle (4,10);
    \fill[color = gray!50] (7,10) rectangle (8,11);
    \fill[color = gray!50] (9,8) rectangle (10,9);
    \fill[color = gray!50] (8,7) rectangle (9,8);
    \fill[color = gray!50] (10,6) rectangle (11,7);
    \fill[color = gray!50] (11,4) rectangle (12,6);
    \trigrid{12};
    \draw[very thick, color = red] (3, 12)--(1, 12)--++(0, -1)--++(1, 0)--++(0, -1)--++(1, 0)--cycle;
    \draw[very thick, color = green] (4, 10)--(2, 10)--++(0, -1)--++(1, 0)--++(0, -1)--++(1, 0)--cycle;
    \draw[very thick, color = orange] (8, 12)--(5, 12)--++(0, -1)--++(1, 0)--++(0, -1)--++(1, 0)--++(0, -1)--++(1, 0)--cycle;
    \draw[very thick, color = cyan] (12, 7)--(9, 7)--++(0, -1)--++(1, 0)--++(0, -1)--++(1, 0)--++(0, -1)--++(1, 0)--cycle;
    \draw[very thick, color = magenta] (8,7) rectangle (9,8);
    \draw[very thick, color = yellow] (9,8) rectangle (10,9);
    \end{scope}
    
    \draw[->, >=stealth, very thick, dashed] (13, 6)--(15.5, 6);

    \begin{scope}[xshift = 15cm]
    \fill[color = gray!50] (1,11) rectangle (2,12);
    \fill[color = gray!50] (2,10) rectangle (3,11);
    \fill[color = gray!50] (5,11) rectangle (7,12);
    \fill[color = gray!50] (2,9) rectangle (4,10);
    \fill[color = gray!50] (7,10) rectangle (8,11);
    \fill[color = gray!50] (9,8) rectangle (10,9);
    \fill[color = gray!50] (8,7) rectangle (9,8);
    \fill[color = gray!50] (9,6) rectangle (10,7);
    \fill[color = gray!50] (10,5) rectangle (11,6);
    \fill[color = gray!50] (11,4) rectangle (12,5);
    \trigrid{12};
    \draw[very thick, color = red] (4, 12)--(0, 12)--++(0, -1)--++(1, 0)--++(0, -1)--++(1, 0)--++(0, -1)--++(1, 0)--++(0, -1)--++(1, 0)--cycle;
    \draw[very thick, color = orange] (8, 12)--(5, 12)--++(0, -1)--++(1, 0)--++(0, -1)--++(1, 0)--++(0, -1)--++(1, 0)--cycle;
    \draw[very thick, color = cyan] (12, 8)--(8, 8)--++(0, -1)--++(1, 0)--++(0, -1)--++(1, 0)--++(0, -1)--++(1, 0)--++(0, -1)--++(1, 0)--cycle;
    \draw[very thick, color = yellow] (9,8) rectangle (10,9);
    \end{scope}
    
    \draw[->, >=stealth, very thick, dashed] (28, 6)--(30.5, 6);

    \begin{scope}[xshift = 30cm]
    \fill[color = gray!50] (1,11) rectangle (2,12);
    \fill[color = gray!50] (2,10) rectangle (3,11);
    \fill[color = gray!50] (5,11) rectangle (7,12);
    \fill[color = gray!50] (2,9) rectangle (4,10);
    \fill[color = gray!50] (7,10) rectangle (8,11);
    \fill[color = gray!50] (9,8) rectangle (10,9);
    \fill[color = gray!50] (8,7) rectangle (12,8);
    \trigrid{12};
    \draw[very thick, color = red] (4, 12)--(0, 12)--++(0, -1)--++(1, 0)--++(0, -1)--++(1, 0)--++(0, -1)--++(1, 0)--++(0, -1)--++(1, 0)--cycle;
    \draw[very thick, color = orange] (8, 12)--(5, 12)--++(0, -1)--++(1, 0)--++(0, -1)--++(1, 0)--++(0, -1)--++(1, 0)--cycle;
    \draw[very thick, color = cyan] (12, 9)--(7, 9)--++(0, -1)--++(1, 0)--++(0, -1)--++(1, 0)--++(0, -1)--++(1, 0)--++(0, -1)--++(1, 0)--++(0, -1)--++(1, 0)--cycle;
    \end{scope}
    
    \draw[->, >=stealth, very thick, dashed] (5, -8)--(7.5, -8);

    \begin{scope}[xshift = 7cm, yshift = -14cm]
    \fill[color = gray!50] (1,11) rectangle (2,12);
    \fill[color = gray!50] (2,10) rectangle (3,11);
    \fill[color = gray!50] (5,11) rectangle (6,12);
    \fill[color = gray!50] (6,10) rectangle (7,11);
    \fill[color = gray!50] (2,9) rectangle (4,10);
    \fill[color = gray!50] (7,9) rectangle (8,10);
    \fill[color = gray!50] (7,8) rectangle (12,9);
    \trigrid{12};
    \draw[very thick, color = red] (4, 12)--(0, 12)--++(0, -1)--++(1, 0)--++(0, -1)--++(1, 0)--++(0, -1)--++(1, 0)--++(0, -1)--++(1, 0)--cycle;
    \draw[very thick, color = cyan] (12, 12)--(4, 12)--++(0, -1)--++(1, 0)--++(0, -1)--++(1, 0)--++(0, -1)--++(1, 0)--++(0, -1)--++(1, 0)--++(0, -1)--++(1, 0)--++(0, -1)--++(1, 0)--++(0, -1)--++(1, 0)--++(0, -1)--++(1, 0)--cycle;
    \end{scope}
    
    \draw[->, >=stealth, very thick, dashed] (20, -8)--(22.5, -8);

    \begin{scope}[xshift = 22cm, yshift = -14cm]
    \fill[color = gray!50] (0,11) rectangle (4,12);
    \fill[color = gray!50] (4,11) rectangle (5,12);
    \fill[color = gray!50] (5,10) rectangle (6,11);
    \fill[color = gray!50] (6,9) rectangle (7,10);
    \fill[color = gray!50] (7,8) rectangle (12,9);
    \trigrid{12};
    \draw[very thick, color = cyan] (12, 12)--(0, 12)--++(0, -1)--++(1, 0)--++(0, -1)--++(1, 0)--++(0, -1)--++(1, 0)--++(0, -1)--++(1, 0)--++(0, -1)--++(1, 0)--++(0, -1)--++(1, 0)--++(0, -1)--++(1, 0)--++(0, -1)--++(1, 0)--++(0, -1)--++(1, 0)--++(0, -1)--++(1, 0)--++(0, -1)--++(1, 0)--++(0, -1)--++(1, 0)--cycle;
    \end{scope}
\end{tikzpicture}
    \caption{An example of how to get back to the line from a random element of its orbit.}
    \label{fig:return}
\end{figure}
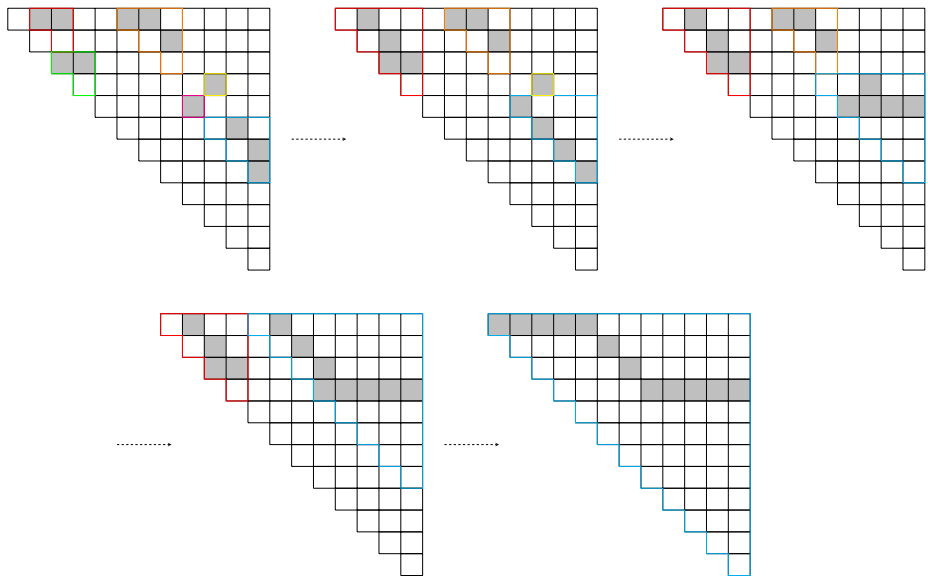

\end{document}